\documentclass{preprint}

\usepackage{amsmath}
\usepackage{paralist}
\usepackage{graphics}
\usepackage{graphicx}
\usepackage{psfrag}

\usepackage[colorlinks=true]{hyperref}

\hypersetup{urlcolor=blue, citecolor=red}

\newtheorem{theorem}{Theorem}[section]
\newtheorem{corollary}{Corollary}

\newtheorem{lemma}[theorem]{Lemma}

\theoremstyle{definition}
\newtheorem{definition}[theorem]{Definition}
\newtheorem{remark}{Remark}
\newtheorem{example}{Example}

\title[Natural Boundary Conditions]{Natural Boundary Conditions\\
in the Calculus of Variations}

\author[A.B. Malinowska and D.F.M. Torres]{}



\email{abmalinowska@ua.pt}
\email{delfim@ua.pt}

\thanks{The first author is on leave of absence
from Faculty of Computer Science,
Bia{\l}ystok Technical University,
15-351 Bia\l ystok, Poland.
E-mail: abmalina@pb.bialystok.pl}

\begin{document}

\maketitle

\centerline{\scshape Agnieszka B. Malinowska}
\medskip
{\footnotesize
 \centerline{Department of Mathematics, University of Aveiro}
   \centerline{3810-193 Aveiro, Portugal}
}

\medskip

\centerline{\scshape Delfim F. M. Torres }
\medskip
{\footnotesize
 \centerline{Department of Mathematics, University of Aveiro}
   \centerline{3810-193 Aveiro, Portugal}
}

\bigskip

\begin{abstract}
We prove necessary optimality conditions for
problems of the calculus of variations on time scales
with a Lagrangian depending on the free end-point.
\end{abstract}


\section{Introduction}

The calculus on time scales was introduced by
Bernd Aulbach and Stefan Hilger in 1988
\cite{Au:Hilger}. The new theory unify and extends
the traditional areas of continuous and discrete analysis
and the various dialects of $q$-calculus \cite{dt:qc}
into a single theory \cite{book:ts1,1st:book:ts},
and is finding numerous applications in such areas
as engineering, biology, economics, finance,
and physics \cite{ts:survey}. The present work is dedicated
to the study of problems of calculus of variations on a generic
time scale $\mathbb{T}$. As particular cases,
one gets the classical calculus of variations \cite{GelfandFomin}
by choosing $\mathbb{T} = \mathbb{R}$;
the discrete-time calculus of variations \cite{KP}
by choosing $\mathbb{T} = \mathbb{Z}$;
and the $q$-calculus of variations
\cite{Bang:q-calc} by choosing $\mathbb{T} =
q^{\mathbb{N}_0}:=\{q^k | k \in \mathbb{N}_0\}$, $q>1$.

The calculus of variations on time scales was born with the works \cite{cv:02} and \cite{b7} and seems
to have interesting applications in Economics
\cite{Almeida:T,Atici:Uysal:06,Atici:Uysal:08,RuiDel07}.
Currently, several researchers are getting interested in the new theory and contributing to its development (see, \textrm{e.g.},
\cite{ZbigDel,bhn:Gus,Rui:Del:HO,HZ1,Mal:Tor:09}).
Here we develop further the theory by proving
necessary optimality conditions for more general
problems of the calculus of variations with a Lagrangian
that may also depend on the unspecified end-point $x(T)$.

In Section~\ref{sec:prm} we review the necessary concepts and tools on time scales; our results are given in Section~\ref{sec:mr}. We begin
Section~\ref{sec:mr} by formulating
the problem \eqref{vp}--\eqref{bc} under study:
to minimize a delta-integral functional subject to a given fixed initial-point $x(a) = \alpha$ and having no constraint on $x(T)$.
The novelty is the dependence of the integrand $f$ on the free end-point $x(T)$. Necessary optimality conditions for such problems, on a general time scale, are given using both Lagrangian (Theorem~\ref{thm:mr}) and Hamiltonian formalisms (Theorem~\ref{thm:mr:ham}). Under appropriate
convexity and linearity assumptions,
the conditions turn out to be
sufficient for a global minimum
(\textrm{cf.} Theorem~\ref{sc}).
A number of important corollaries are obtained,
and several examples illustrating the new results
discussed in detail. Corollary~\ref{cor:anws:AZ:CV} (see also Corollary~\ref{Cor:T=R:Hamilt})
give answer to a question posed
to the second author by A.~Zinober in May 2008 during
a visit to the University of Aveiro,
and again presented as an open question during
the conference ``Calculus of Variations
and Applications---from Engineering to Economy'',
held from 8th to 10th September 2008 in the
New University of Lisbon, Monte de Caparica, Portugal:
``What are the necessary optimality conditions for the
problem of the calculus of variations with
a free end-point $x(T)$ but whose Lagrangian
depends explicitly on $x(T)$?'' The new transversality
condition \eqref{eq:tc} (or the equivalent
natural boundary condition \eqref{bc:R})
seems to have important implications in Economics.
This question is under study by
Alan Zinober, Kim Kaivanto, and Pedro Cruz
and will appear elsewhere.


\section{Preliminaries}
\label{sec:prm}

In this section we introduce basic definitions and results that
will be needed for the rest of the paper. For a more general
presentation of the theory of time scales, we refer the reader
to \cite{BP}.

A nonempty closed subset of $\mathbb{R}$ is called a \emph{time
scale} and it is denoted by $\mathbb{T}$. Thus, $\mathbb{R}$,
$\mathbb{Z}$, and $\mathbb{N}$, are trivial examples of times
scales. Other examples of times scales are: $[-2,4] \bigcup
\mathbb{N}$, $h\mathbb{Z}:=\{h z | z \in \mathbb{Z}\}$ for some
$h>0$, $q^{\mathbb{N}_0}:=\{q^k | k \in \mathbb{N}_0\}$ for some
$q>1$, and the Cantor set. We assume that a time scale $\mathbb{T}$
has the topology that it inherits from the real numbers with the
standard topology.

The \emph{forward jump operator}
$\sigma:\mathbb{T}\rightarrow\mathbb{T}$ is defined by
$$\sigma(t)=\inf{\{s\in\mathbb{T}:s>t}\},
\mbox{ for all $t\in\mathbb{T}$},$$ while the \emph{backward jump
operator} $\rho:\mathbb{T}\rightarrow\mathbb{T}$ is defined by
$$\rho(t)=\sup{\{s\in\mathbb{T}:s<t}\},\mbox{ for all
$t\in\mathbb{T}$},$$ with $\inf\emptyset=\sup\mathbb{T}$
(\textrm{i.e.}, $\sigma(M)=M$ if $\mathbb{T}$ has a maximum $M$) and $\sup\emptyset=\inf\mathbb{T}$ (\textrm{i.e.}, $\rho(m)=m$ if
$\mathbb{T}$ has a minimum $m$).

A point $t\in\mathbb{T}$ is called \emph{right-dense},
\emph{right-scattered}, \emph{left-dense} and \emph{left-scattered}
if $\sigma(t)=t$, $\sigma(t)>t$, $\rho(t)=t$ and $\rho(t)<t$,
respectively.

The \emph{graininess function} $\mu:\mathbb{T}\rightarrow[0,\infty)$
is defined by
$$\mu(t)=\sigma(t)-t,\mbox{ for all $t\in\mathbb{T}$}.$$

\begin{example}
If $\mathbb{T} = \mathbb{R}$, then $\sigma(t) = \rho(t) = t$
and $\mu(t)= 0$. If $\mathbb{T} = \mathbb{Z}$, then
$\sigma(t) = t + 1$, $\rho(t) = t - 1$, and
$\mu(t)= 1$. On the other hand, if
$\mathbb{T} = q^{\mathbb{N}_0}$, where $q>1$ is a fixed real number, then we have $\sigma(t) = q t$, $\rho(t) = q^{-1} t$, and
$\mu(t)= (q-1) t$.
 \end{example}

\begin{definition} \cite{BP}
A time scale $\mathbb{T}$ is called \emph{regular} if the
following two conditions are satisfied simultaneously:
\begin{itemize}
\item[(i)] $\sigma(\rho(t))=t$, for all $t\in \mathbb{T}$;
\item[(ii)] $\rho(\sigma((t))=t$, for all $t\in \mathbb{T}$.
\end{itemize}
\end{definition}

Following \cite{BP}, let us define
$\mathbb{T}^{\kappa}=\mathbb{T}\backslash(\rho(b),b]$.

\begin{definition}
\label{def:de:dif}
We say that a function $f:\mathbb{T}\rightarrow\mathbb{R}$ is
\emph{delta differentiable} at $t\in\mathbb{T}^{\kappa}$ if there
exists a number $f^{\Delta}(t)$ such that for all $\varepsilon>0$
there is a neighborhood $U$ of $t$ (\textrm{i.e.},
$U=(t-\delta,t+\delta)\cap\mathbb{T}$ for some $\delta>0$) such that
$$|f(\sigma(t))-f(s)-f^{\Delta}(t)(\sigma(t)-s)|
\leq\varepsilon|\sigma(t)-s|,\mbox{ for all $s\in U$}.$$ We call
$f^{\Delta}(t)$ the \emph{delta derivative} of $f$ at $t$ and say
that $f$ is \emph{delta differentiable} on $\mathbb{T}^{\kappa}$
provided $f^{\Delta}(t)$ exists for all
$t\in\mathbb{T}^{\kappa}$.
\end{definition}

\begin{remark}
If $t \in \mathbb{T} \setminus \mathbb{T}^\kappa$, then
$f^{\Delta}(t)$ is not uniquely defined, since for such
a point $t$, small neighborhoods $U$ of $t$ consist only
of $t$ and, besides, we have $\sigma(t) = t$. For this reason,
maximal left-scattered points are omitted in Definition~\ref{def:de:dif}.
\end{remark}

Note that in right-dense points
$f^{\Delta} (t)=lim_{s\rightarrow t}=\frac{f(t)-f(s)}{t-s}$,
provided this limit exists, and in right-scattered points
$f^{\Delta} (t)=\frac{f(\sigma(t))-f(t)}{\mu(t)}$,
provided $f$ is continuous at $t$.

\begin{example}
\label{ex:der:QC}
If $\mathbb{T} = \mathbb{R}$, then $f^{\Delta}(t) = f'(t)$, \textrm{i.e.}, the delta derivative coincides with the usual one.
If $\mathbb{T} = \mathbb{Z}$, then $f^{\Delta}(t) = \Delta f(t)
= f(t+1) - f(t)$. If $\mathbb{T} = q^{\mathbb{N}_0}$, $q>1$,
then $f^{\Delta} (t)=\frac{f(q t)-f(t)}{(q-1) t}$,
\textrm{i.e.}, we get the usual derivative of Quantum calculus \cite{QC}.
\end{example}

A function $f:\mathbb{T}\rightarrow\mathbb{R}$ is called
\emph{rd-continuous} if it is continuous at right-dense points and
if its left-sided limit exists at left-dense points. We denote the
set of all rd-continuous functions by C$_{\textrm{rd}}$ and the set
of all delta differentiable functions with rd-continuous derivative by C$_{\textrm{rd}}^1$.
It is known that  rd-continuous functions possess an
\emph{antiderivative}, \textrm{i.e.}, there exists a function $F$
with $F^{\Delta}=f$, and in this case the \emph{delta integral}
is defined by $\int_{c}^{d}f(t)\Delta t=F(d)-F(c)$ for all
$c,d\in\mathbb{T}$.

\begin{example}
Let $a, b \in \mathbb{T}$ with $a < b$.
If $\mathbb{T} = \mathbb{R}$, then
$\int_{a}^{b}f(t)\Delta t = \int_{a}^{b}f(t) dt$,
where the integral on the right-hand side is the
classical Riemann integral. If $\mathbb{T} = \mathbb{Z}$,
then $\int_{a}^{b}f(t)\Delta t = \sum_{k=a}^{b-1} f(k)$.
If $\mathbb{T} = q^{\mathbb{N}_0}$, $q>1$, then
$\int_{a}^{b}f(t)\Delta t = (1 - q) \sum_{t \in [a,b)} t f(t)$.
\end{example}

The delta integral has the following properties:
\begin{itemize}

\item[(i)] if $f\in C_{rd}$ and $t \in \mathbb{T}^{\kappa}$, then
\begin{equation*}
\int_t^{\sigma(t)}f(\tau)\Delta\tau=\mu(t)f(t)\, ;
\end{equation*}

\item[(ii)]if $c,d\in\mathbb{T}$ and $f,g\in C_{rd}$, then

\begin{equation*}
\int_{c}^{d}f(\sigma(t))g^{\Delta}(t)\Delta t
=\left[(fg)(t)\right]_{t=c}^{t=d}-\int_{c}^{d}f^{\Delta}(t)g(t)\Delta
t \, ,
\end{equation*}

\begin{equation*}
\int_{c}^{d}f(t)g^{\Delta}(t)\Delta t
=\left[(fg)(t)\right]_{t=c}^{t=d}-\int_{c}^{d}
f^{\Delta}(t)g(\sigma(t))\Delta t.
\end{equation*}

\end{itemize}

The Dubois-Reymond lemma of the calculus of variations
on time scales will be useful for our purposes.

\begin{lemma}(Lemma of Dubois-Reymond \cite{b7})
\label{lemma:DR} Let $g\in C_{\textrm{rd}}$,
$g:[a,b]^k\rightarrow\mathbb{R}$. Then,
$$\int_{a}^{b}g(t) \cdot \eta^\Delta(t)\Delta t=0  \quad
\mbox{for all $\eta\in C_{\textrm{rd}}^1$ with
$\eta(a)=\eta(b)=0$}$$ if and only if $g(t)=c \mbox{ on $[a,b]^k$
for some $c\in\mathbb{R}$}$.
\end{lemma}


\section{Main Results}
\label{sec:mr}

Let $\mathbb{T}$ be a bounded time scale. Throughout we let
$A,B\in \mathbb{T}$ with $A<B$. For an interval $[c,d]\cap \mathbb{T}$ we simply write $[c,d]$. We
also abbreviate $f\circ\sigma$ by $f^\sigma$.
Now let $[a,T]$ with $T<B$ be a subinterval of $[A,B]$. The problem of the calculus of variations on time scales under consideration has the form
\begin{equation}\label{vp}
 \text{minimize} \quad   \mathcal{L}[x]=\int_{a}^{T}f(t,x^{\sigma}(t),x^{\Delta}(t),x(T))\Delta
     t ,
\end{equation}
over all $x\in C_{rd}^{1}$ satisfying the boundary condition
\begin{equation}\label{bc}
    x(a)=\alpha\, , \quad \alpha\in \mathbb{R}
    \quad (x(T) \text{ free}),
\end{equation}
where the function $(t,x,v,z)\rightarrow f(t,x,v,z)$ from
$[a,T]\times \mathbb{R} \times \mathbb{R} \times \mathbb{R} $ to  $\mathbb{R}$ has partial
continuous derivatives with respect to $x,v,z$ for all $t\in[a,T]$,
and $f(t,\cdot,\cdot,\cdot)$ and its partial derivatives are
rd-continuous for all $t$.
A function $x\in C_{rd}^{1}$ is said to be admissible if it is
satisfies condition \eqref{bc}.

Let us consider the following norm in $C_{rd}^{1}$:
\begin{equation*}
   \|x\|_{1}=sup_{t\in[a,T]}|x^{\sigma}(t)|+sup_{t\in[a,T]}|x^{\Delta}(t)|.
\end{equation*}

\begin{definition}
An admissible function $\tilde{x}$ is said to be a
\emph{weak local minimum} for \eqref{vp}--\eqref{bc} if there exists $\delta >0$ such
that $\mathcal{L}[\tilde{x}]\leq \mathcal{L}[x]$ for all admissible
$x$ with $\|x-\tilde{x}\|_{1}<\delta$.
\end{definition}


\subsection{Lagrangian approach}
\label{sub:sec:Lag}

Next theorem gives necessary optimality conditions
for problem \eqref{vp}--\eqref{bc}.

\begin{theorem}
\label{thm:mr} If $\tilde{x}(\cdot)$ is a solution of the problem
\eqref{vp}--\eqref{bc}, then
\begin{equation}\label{Euler}
f_{x^{\Delta}}^{\Delta}(t,\tilde{x}^{\sigma}(t),\tilde{x}^{\Delta}(t),\tilde{x}(T))=
f_{x^{\sigma}}(t,\tilde{x}^{\sigma}(t),\tilde{x}^{\Delta}(t),\tilde{x}(T))
\end{equation}
for all $t \in [a,T]^\kappa$. Moreover,
\begin{multline}\label{new:tc}
f_{x^{\Delta}}(\rho(T),\tilde{x}^{\sigma}(\rho
(T)),\tilde{x}^{\Delta}(\rho(T)),\tilde{x}(T))+\int_{\rho(T)}^Tf_{x^{\sigma}}(t,\tilde{x}^{\sigma}(t),\tilde{x}^{\Delta}(t),\tilde{x}(T))\Delta
t\\
+\int_a^T
f_{z}(t,\tilde{x}^{\sigma}(t),\tilde{x}^{\Delta}(t),\tilde{x}(T))\Delta
t = 0.
\end{multline}
\end{theorem}

\begin{proof}
Suppose that $\mathcal{L}[\cdot]$ has a weak local minimum at
$\tilde{x}(\cdot)$. We can proceed as Lagrange did, by considering the value of $\mathcal{L}$ at a nearby function
$x = \tilde{x} + \varepsilon
h$, where $\varepsilon\in \mathbb{R}$ is a small parameter, $h(\cdot) \in
C^1_{rd}$, and $h(a) = 0$. Because $x(T)$ is free, we do not require
$h(\cdot)$ to vanish at $T$. Let
$$
\phi(\varepsilon) = \mathcal{L}[(\tilde{x} + \varepsilon h)(\cdot)] = \int_a^T
f(t,\tilde{x}^{\sigma}(t)+\varepsilon
h^{\sigma}(t),\tilde{x}^{\Delta}(t)+ \varepsilon
h^{\Delta}(t),\tilde{x}(T)+\varepsilon h(T)) \Delta t.
$$
A necessary condition for $\tilde{x}(\cdot)$ to be a minimum  is
given by
\begin{equation}
\label{eq:FT} \left.\phi'(\varepsilon)\right|_{\varepsilon=0} = 0
\Leftrightarrow \int_a^T \left[ f_{x^{\sigma}}(\cdots) h^{\sigma}(t)
+ f_{x^{\Delta}}(\cdots) h^{\Delta}(t) + f_z(\cdots)
h(T)\right]\Delta t = 0 \, ,
\end{equation}
where $(\cdots) =
\left(t,\tilde{x}^{\sigma}(t),\tilde{x}^{\Delta}(t),\tilde{x}(T)\right)$.
Integration by parts gives
\begin{equation*}
\int_a^T f_{x^{\sigma}}(\cdots) h^{\sigma}(t) \Delta t =\int_a^t
f_{x^{\sigma}}(\cdots) \Delta \tau h(t)|_{t=a}^{t=T}-\int_a^T
\left(\int_a^t f_{x^{\sigma}}(\cdots)\Delta \tau
h^{\Delta}(t)\right) \Delta t.
\end{equation*}
Because $h(a) = 0$, the necessary condition \eqref{eq:FT} can be
written as
\begin{multline}
\label{eq:aft:IP} 0 = \int_a^T
\left(f_{x^{\Delta}}(\cdots)-\int_a^t
f_{x^{\sigma}}(\cdots)\Delta \tau \right)  h^{\Delta}(t) \Delta t \\
+\int_a^T f_{x^{\sigma}}(\cdots)\Delta \tau h(T)+\int_a^T
f_{z}(\cdots)\Delta t h(T)
\end{multline}
for all $h(\cdot) \in C_{rd}^1$ such that $h(a) = 0$. In particular,
equation \eqref{eq:aft:IP} holds for the subclass of functions
$h(\cdot) \in C_{rd}^1$ that do vanish at $h(T)$. Thus, by the
Dubois-Reymond Lemma~\ref{lemma:DR}, we have
\begin{equation}
\label{eq:EL} f_{x^{\Delta}}(\cdots)-\int_a^t
f_{x^{\sigma}}(\cdots)\Delta \tau=c,
\end{equation}
for some $c\in \mathbb{R}$ and all $t\in[a,T]$. Equation \eqref{eq:aft:IP}
must be satisfied for all $h(\cdot) \in C_{rd}^1$ with $h(a) = 0$,
which includes functions $h(\cdot)$ that do not vanish at $T$.
Consequently, equations \eqref{eq:aft:IP} and \eqref{eq:EL}
imply that
\begin{equation}\label{eq1}
c+\int_a^Tf_{x^{\sigma}}(\cdots)\Delta t+\int_a^T
f_{z}(\cdots)\Delta t = 0.
\end{equation}
From the properties of the delta integral and from \eqref{eq:EL}, it follows that
\begin{multline*} c+\int_a^Tf_{x^{\sigma}}(\cdots)\Delta
t=c+\int_a^{\rho(T)}f_{x^{\sigma}}(\cdots)\Delta
t+\int_{\rho(T)}^Tf_{x^{\sigma}}(\cdots)\Delta t\\
=f_{x^{\Delta}}(\rho(T),\tilde{x}^{\sigma}(\rho
(T)),\tilde{x}^{\Delta}(\rho(T)),\tilde{x}(T))+\int_{\rho(T)}^Tf_{x^{\sigma}}(\cdots)\Delta
t \, .
\end{multline*}
Hence, we can rewrite \eqref{eq1} as \eqref{new:tc}.
\end{proof}

\begin{theorem}
\label{thm:mr:reg}
Let $\mathbb{T}$ be a regular time scale. If $\tilde{x}(\cdot)$ is a
solution of the problem \eqref{vp}--\eqref{bc}, then
\begin{equation*}
f_{x^{\Delta}}^{\Delta}(t,\tilde{x}^{\sigma}(t),\tilde{x}^{\Delta}(t),\tilde{x}(T))=
f_{x^{\sigma}}(t,\tilde{x}^{\sigma}(t),\tilde{x}^{\Delta}(t),\tilde{x}(T))
\end{equation*}
for all $t \in [a,T]^\kappa$. Moreover,
\begin{multline}\label{new:tc:reg}
f_{x^{\Delta}}(\rho(T),\tilde{x}^{\sigma}(\rho
(T)),\tilde{x}^{\Delta}(\rho(T)),\tilde{x}(T))\\
+\mu (\rho(T))f_{x^{\sigma}}(\rho (T)),\tilde{x}^{\sigma}(\rho
(T)),\tilde{x}^{\Delta}(\rho (T)),\tilde{x}(T))\\
+ \int_a^T
f_{z}(t,\tilde{x}^{\sigma}(t),\tilde{x}^{\Delta}(t),\tilde{x}(T))\Delta
t = 0.
\end{multline}
\end{theorem}
\begin{proof}
By Theorem~\ref{thm:mr} we need only to show that on a regular time scale equation \eqref{new:tc} can be written in the form
\eqref{new:tc:reg}. Indeed, from the properties of the delta integral it follows that
\begin{equation*}
\int_{\rho(T)}^Tf_{x^{\sigma}}(\cdots)\Delta t =\mu (\rho
(T))f_{x^{\sigma}}(\rho (T)),\tilde{x}^{\sigma}(\rho
(T)),\tilde{x}^{\Delta}(\rho (T)),\tilde{x}(T)) \, .
\end{equation*}
\end{proof}

Choosing $\mathbb{T}=\mathbb{R}$ in Theorem~\ref{thm:mr:reg}
we immediately obtain the corresponding
result in the classical context of the
calculus of variations. We were not able to find a reference,
in the vast and rich literature of the calculus of variations,
to the result given by Corollary~\ref{cor:anws:AZ:CV}.

\begin{corollary}
\label{cor:anws:AZ:CV}
If $\tilde{x}(\cdot)$ is a solution of the
problem
\begin{equation*}
\begin{gathered}
 \text{minimize} \quad   \mathcal{L}[x]
 = \int_a^T f(t,x(t),x'(t),x(T)) dt \\
x(a) = \alpha \, , \quad \alpha\in \mathbb{R}
\quad (x(T) \text{ free}),
\end{gathered}
\end{equation*}
where $a,T \in \mathbb{R}$, $a<T$, $x(\cdot) \in C^1$, then
the Euler-Lagrange equation
\begin{equation*}
\frac{d}{dt}
f_{x'}\left(t,\tilde{x}(t),\tilde{x}'(t),\tilde{x}(T)\right) =
f_x\left(t,\tilde{x}(t),\tilde{x}'(t),\tilde{x}(T)\right)
\end{equation*}
holds for all $t \in [a,T]$. Moreover,
\begin{equation}\label{bc:R}
f_{x'}\left(T,\tilde{x}(T),\tilde{x}'(T), \tilde{x}(T)\right) = -
\int_a^T f_z\left(t,\tilde{x}(t),\tilde{x}'(t),\tilde{x}(T)\right)
dt \, .
\end{equation}
\end{corollary}

\begin{remark}
In the classical setting $f$ does not depend on $x(T)$, \textrm{i.e.}, $f_z = 0$. Then, \eqref{bc:R}
reduces to the well known natural boundary condition
$f_{x'}\left(T,\tilde{x}(T),\tilde{x}'(T)\right) = 0$.
\end{remark}

Similarly, we can obtain other corollaries
by choosing different time scales. Corollary~\ref{cor:Z}
is obtained from Theorem~\ref{thm:mr:reg}
letting $\mathbb{T}=\mathbb{Z}$.

\begin{corollary}
\label{cor:Z}
If $\tilde{x}(\cdot)$ is a solution of the
discrete-time problem
\begin{equation*}
\begin{gathered}
 \text{minimize} \quad   \mathcal{L}[x]
 = \sum _{t=a}^{T-1}f(t,x(t+1),\Delta x(t),x(T))\\
x(a) = \alpha \, , \quad \alpha\in \mathbb{R}
\quad (x(T) \text{ free}),
\end{gathered}
\end{equation*}
where $a,T \in\mathbb{Z}$, $a<T$, then
\begin{equation*}
f_{x}\left(t,\tilde{x}(t+1),\Delta\tilde{x}(t),\tilde{x}(T)\right)
= \Delta f_{\Delta
x}\left(t,\tilde{x}(t+1),\Delta\tilde{x}(t),\tilde{x}(T)\right)
\end{equation*}
for all $t \in [a,T-1]$. Moreover,
\begin{multline}\label{bc:Z}
f_{x}\left(T-1,\tilde{x}(T),\Delta\tilde{x}(T-1),\tilde{x}(T)\right)
+f_{\Delta x}\left(T-1,\tilde{x}(T),\Delta\tilde{x}(T-1),\tilde{x}(T)\right)\\
=-\sum _{t=a}^{T-1}f_z(t,\tilde{x}(t+1),\Delta
\tilde{x}(t),\tilde{x}(T)) \, .
\end{multline}
\end{corollary}

\begin{remark}
In the case $f$ does not depend on $x(T)$, \eqref{bc:Z}
reduces to the natural boundary condition for the discrete variational problem (see \cite[Theorem 8.3]{KP}).
\end{remark}

Let now $\mathbb{T}=q^{\mathbb{N}_{0}}$, $q>1$.
We then obtain the analogous result for
the $q$-calculus of variations. In what
follows we use the standard notation
$D_q$ for the $q$-derivative:
\begin{equation}
\label{eq:q-deriv}
D_q x(t) = \frac{x(qt)-x(t)}{qt-t}
\end{equation}
(\textrm{cf.} Example~\ref{ex:der:QC}).
The $q$-derivative \eqref{eq:q-deriv}
is also known in the literature
as the Jackson derivative \cite{Jackson}.

\begin{corollary}
If $\tilde{x}(\cdot)$ is
a solution of the problem
\begin{equation*}
\begin{gathered}
\text{minimize} \quad   \mathcal{L}[x]
= \sum _{t=a}^{Tq^{-1}}(q-1)t
f \left(t,x(qt),D_q x(t),x(T)\right)\\
x(a) = \alpha \, , \quad \alpha\in \mathbb{R}
\quad (x(T) \text{ free}) \, ,
\end{gathered}
\end{equation*}
where $a,T \in\mathbb{T} $, $a<T$, then
\begin{equation*}
f_{x}\left(t,\tilde{x}(q t),D_q \tilde{x}(t),\tilde{x}(T)\right)
= D_q f_{v}\left(t,\tilde{x}(qt),D_q \tilde{x}(t),\tilde{x}(T)\right)
\end{equation*}
for all $t \in \left[a,Tq^{-1}\right]$. Moreover,
\begin{multline*}
f_{v}\left(Tq^{-1},\tilde{x}(T),
D_q \tilde{x}\left(Tq^{-1}\right),\tilde{x}(T)\right)\\
+T(1-q^{-1})f_{x}\left(Tq^{-1},\tilde{x}(T),D_q \tilde{x}\left(Tq^{-1}\right),\tilde{x}(T)\right)\\
=-\sum_{t=a}^{Tq^{-1}}f_z
\left(t,\tilde{x}(qt),D_q \tilde{x}(t),\tilde{x}(T)\right).
\end{multline*}
\end{corollary}

We illustrate the application
of our Theorem~\ref{thm:mr} with an example.

\begin{example}\label{fromB}
Consider the problem
\begin{equation}\label{AD}
 \text{minimize} \quad   \mathcal{L}[x]=\int_{0}^{1}\left(\sqrt{1+(x^{\Delta}(t))^2}+\beta(x(1)-1)^2 \right)\Delta
     t ,
\end{equation}
where $\beta \in \mathbb{R}^{+}$, subject to the boundary condition
\begin{equation}
\label{eq:bc}
    x(0)=0  \quad (x(1) \text{ free}).
\end{equation}
Since
\begin{equation*}
f(t,x^{\sigma},x^{\Delta},z)=\sqrt{1+(x^{\Delta})^2}+\beta(z-1)^2,
\end{equation*}
we have
\begin{equation*}
\begin{split}
f_{x^{\sigma}}(t,x^{\sigma},x^{\Delta},z)&=0, \\
f_{x^{\Delta}}(t,x^{\sigma},x^{\Delta},z)&=\frac{x^{\Delta}}{\sqrt{1+(x^{\Delta})^2}},\\
f_z(t,x^{\sigma},x^{\Delta},z)&=2\beta(z-1).
\end{split}
\end{equation*}
If $\tilde{x}(\cdot)$ is a local minimizer of
\eqref{AD}--\eqref{eq:bc}, then conditions
\eqref{Euler}--\eqref{new:tc} must hold, \textrm{i.e.},
\begin{equation}\label{AD1}
f_{x^{\Delta}}^{\Delta}(t,\tilde{x}^{\sigma}(t),\tilde{x}^{\Delta}(t),\tilde{x}(T))=0,
\end{equation}
\begin{equation}\label{AD2}
f_{x^{\Delta}}(\rho(1),\tilde{x}^{\sigma}(\rho(1)),\tilde{x}^{\Delta}(\rho(1)),\tilde{x}(1))=-\int_{0}^{1}2\beta(\tilde{x}(1)-1)\Delta
t.
\end{equation}
Equation \eqref{AD1} implies that there exists a constant $d\in
\mathbb{R}$ such that
\begin{equation*}
\tilde{x}^{\Delta}(t)=d\sqrt{1+(\tilde{x}^{\Delta}(t))^2}.
\end{equation*}
Solving the latter equation with initial condition $\tilde{x}(0)=0$
we obtain $\tilde{x}(t)=\alpha t$, where $\alpha \in \mathbb{R}$. In order to determine $\alpha$ we use the natural boundary condition \eqref{AD2}, which can be rewritten as
\begin{equation}\label{AD3}
\frac{\alpha}{\sqrt{1+\alpha^{2}}}=-2\beta(\alpha-1).
\end{equation}
The real solution of equation \eqref{AD3} is

\begin{equation*}
\alpha(\beta) = {\frac {2\,{\beta}^{2}+\sqrt
{4\,{\beta}^{4}+{\beta}^{2}}}{{4 \beta }^{2}}}-{\frac {\sqrt
{-8\,{\beta}^{2}+4\,\sqrt {4\,{\beta}^{4}+{ \beta}^{2}}+1}}{4
\beta}}.
\end{equation*}
Hence, $\tilde{x}(t)=\alpha(\beta) t$
is a candidate to be a minimizer with
\begin{multline*}
\mathcal{L}[\tilde{x}] = \sqrt {1+ \left( {\frac
{2\,{\beta}^{2}+\sqrt {4\,{\beta}^{4}+{ \beta}^{2}}}{{4
\beta}^{2}}}- {\frac {\sqrt {-8\,{\beta}^{2}+4\,
\sqrt {4\,{\beta}^{4}+{\beta}^{2}}+1}}{4 \beta}} \right) ^{2}}\\
+\beta\,
 \left( {\frac {2\,{\beta}^{2}+\sqrt {4\,{\beta}^{4}+{\beta}^{2}}
}{{4 \beta}^{2}}}-{\frac {\sqrt {-8\,{\beta}^{2}+4\,\sqrt {4\,{
\beta}^{4}+{\beta}^{2}}+1}}{4 \beta}}-1 \right) ^{2}.
\end{multline*}
The extremal $\tilde{x}(t) = \alpha(\beta) t$ is represented in
Figure~\ref{fig:1} for different values of $\beta$.
\begin{figure}
\label{fig:1}
\begin{center}
\psfrag{beta=1}{{\tiny $\beta=1$}}
\psfrag{beta=2}{{\tiny $\beta=2$}}
\psfrag{beta=15}{{\tiny $\beta=15$}}
\psfrag{beta=inf}{{\tiny $\beta=\infty$}}
\psfrag{1}{\hspace*{-0.1cm}{\tiny 1}}
\psfrag{0.8}{\hspace*{-0.1cm}{\tiny 0.8}}
\psfrag{0.6}{\hspace*{-0.1cm}{\tiny 0.6}}
\psfrag{0.4}{\hspace*{-0.1cm}{\tiny 0.4}}
\psfrag{0.2}{\hspace*{-0.1cm}{\tiny 0.2}}
\psfrag{0}{\hspace*{-0.1cm}{\tiny 0}}
\psfrag{t}{{\tiny $t$}}
\includegraphics[width=12cm,height=8cm,angle=-90]{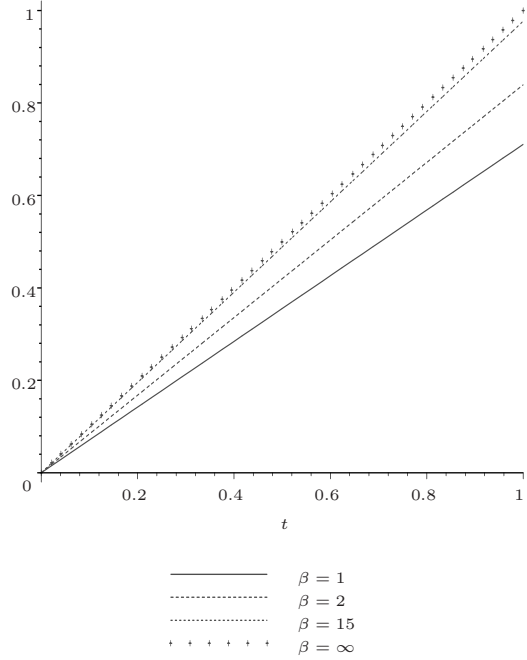}
\end{center}
\vspace*{-2cm}
\caption{The extremal $\tilde{x}(t) = \alpha(\beta) t$
of Example~\ref{fromB} for different values of the parameter $\beta$.}
\end{figure}
We note  that
$$\lim_{\beta\rightarrow \infty} \alpha(\beta)= 1,$$
and in the limit, when $\beta = \infty$, the solution of
\eqref{AD}--\eqref{eq:bc} coincides with the solution of the
following problem with fixed initial and terminal points
(\textrm{cf.} \cite{b7}):
\begin{equation*}
\begin{gathered}
 \text{minimize} \quad   \mathcal{L}[x]
 =\int_{0}^{1}\left(\sqrt{1+(x^{\Delta}(t))^2}\right)\Delta t \\
 x(0) = 0 \, , \quad x(1) = 1 \, .
\end{gathered}
\end{equation*}
\end{example}


\subsection{Hamiltonian approach}
\label{sub:sec:Ham}

Hamiltonian systems on time scales
were introduced in \cite{b1} and have a central role
in the study of optimal control problems on time scales \cite{HZ2}. Let us consider now the more general problem
\begin{equation}\label{ocp}
 \text{minimize} \quad   \mathcal{L}[x,u]=\int_{a}^{T}f(t,x^{\sigma}(t),u^{\sigma}(t),x(T))\Delta
     t
\end{equation}
subject to
\begin{equation}\label{ocbc:1}
x^{\Delta}(t)=g(t,x^{\sigma}(t),u^{\sigma}(t),x(T)),
\end{equation}
\begin{equation}\label{ocbc:2}
x(a)=\alpha\, ,   \quad \alpha\in \mathbb{R}
\quad (x(T) \text{ free}),
\end{equation}
where $f(t,x,v,z):[a,T]\times \mathbb{R} \times \mathbb{R} \times \mathbb{R} \rightarrow \mathbb{R}$
and $g(t,x,v,z):[a,T]\times \mathbb{R} \times \mathbb{R} \times \mathbb{R} \rightarrow \mathbb{R}$
have partial continuous derivatives with respect to $x,v,z$ for all $t\in[a,T]$, and $f(t,\cdot,\cdot,\cdot)$, $g(t,\cdot,\cdot,\cdot)$
and their partial derivatives are rd-continuous for all $t$.
In the particular case $g(t,x,v,z) = v$ problem \eqref{ocp}--\eqref{ocbc:2} reduces to \eqref{vp}--\eqref{bc}.

\begin{theorem}
\label{thm:mr:ham}
If $(\tilde{x}(\cdot),\tilde{u}(\cdot))$ is a weak local minimizer for the problem \eqref{ocp}--\eqref{ocbc:2}, then there is a
function $\tilde{\lambda}(\cdot)$ such that the triple
$(\tilde{x}(\cdot),\tilde{u}(\cdot),\tilde{\lambda}(\cdot))$
satisfies: (i) the Hamiltonian system
\begin{gather}
\label{ham2}
x^{\Delta}(t) =
H_{\lambda^{\sigma}}(t,x^{\sigma}(t),u^{\sigma}(t),\lambda^{\sigma}(t),x(T)),\\
\label{ham1}
(\lambda^{\sigma}(t))^{\Delta} =
-H_{x^{\sigma}}(t,x^{\sigma}(t),u^{\sigma}(t),\lambda^{\sigma}(t),x(T)),
\end{gather}
(ii) the stationary condition
\begin{equation}\label{ham3}
H_{u^{\sigma}}(t,x^{\sigma}(t),u^{\sigma}(t),\lambda^{\sigma}(t),x(T))=0,
\end{equation}
for all $t \in [a,T]^\kappa$; and (iii) the transversality condition
\begin{multline}\label{ham4}
\lambda^{\sigma}(\rho(T))
=\int_{\rho(T)}^{T}H_{x^{\sigma}}(t,x^{\sigma}(t),u^{\sigma}(t),\lambda^{\sigma}(t),x(T))\Delta
t\\
+\int_{a}^{T}H_{z}(t,x^{\sigma}(t),u^{\sigma}(t),\lambda^{\sigma}(t),x(T))\Delta
t,
\end{multline}
where the Hamiltonian
$H(t,x,v,\lambda,z):[a,T]\times \mathbb{R} \times \mathbb{R} \times \mathbb{R} \times \mathbb{R}\rightarrow \mathbb{R}$ is
defined by
\begin{equation}\label{hamil}
H(t,x^{\sigma},u^{\sigma},\lambda^{\sigma},z)
=f(t,x^{\sigma},u^{\sigma},z)+\lambda^{\sigma}g(t,x^{\sigma},u^{\sigma},z).
\end{equation}
\end{theorem}

\begin{remark}
In Theorem~\ref{thm:mr:ham} we are assuming to have
a time scale $\mathbb{T}$ for which
$\lambda^{\sigma}(t)$ is delta-differentiable on $[a,T]^\kappa$.
Examples of time scales for which $\sigma$ is not delta-differentiable are easily found \cite{BP}.
\end{remark}

\begin{proof}
Using the Lagrange multiplier rule we can form an expression
$\lambda^{\sigma}(g-x^{\Delta})$ for each value of $t$. The
replacement of $f$ by $f+\lambda^{\sigma}(g-x^{\Delta})$
in the objective functional give us the following new problem:
\begin{multline}\label{new:ocp}
\text{minimize} \quad \mathcal{I}[x,u,\lambda]
=\int_{a}^{T}\Bigl\{f(t,x^{\sigma}(t),u^{\sigma}(t),x(T))\\
+\lambda^{\sigma}(t)\left[g(t,x^{\sigma}(t),u^{\sigma}(t),x(T))-x^{\Delta}(t)\right]\Bigr\}\Delta t ,
\end{multline}
subject to
\begin{equation}\label{new:ocbc}
x(a)=\alpha  \quad (x(T) \text{ free}).
\end{equation}
Suppose that $(\tilde{x}(\cdot),\tilde{u}(\cdot))$ is a weak local minimizer for the problem \eqref{ocp}--\eqref{ocbc:2}. Then the
triple $(\tilde{x}(\cdot),\tilde{u}(\cdot),\tilde{\lambda}(\cdot))$
should be a weak local minimizer for the problem
\eqref{new:ocp}--\eqref{new:ocbc}. Using \eqref{hamil} in
\eqref{new:ocp} we write the  functional in the form
\begin{equation}\label{ham:ocp}
\mathcal{I}[x,u,\lambda]=
\int_{a}^{T}[H(t,x^{\sigma},u^{\sigma},\lambda^{\sigma},x(T))
-\lambda^{\sigma}(t)x^{\Delta}(t)]\Delta t .
\end{equation}
Applying Theorem \ref{thm:mr} to the problem
\eqref{new:ocp}--\eqref{new:ocbc}, in view of \eqref{ham:ocp}, gives conditions \eqref{ham2}--\eqref{ham4}.
\end{proof}

\begin{remark}
If $\mathbb{T}$ is a regular time scale, then by Theorem
\ref{thm:mr:reg} the transversality condition \eqref{ham4} can be
written in the form
\begin{multline*}
\lambda^{\sigma}(\rho(T))=\mu(\rho(T))H_{x^{\sigma}}(\rho(T),x^{\sigma}(\rho(T)),
u^{\sigma}(\rho(T)),\lambda^{\sigma}(\rho(T)),x(T))\\
+\int_{a}^{T}H_{z}(t,x^{\sigma}(t),u^{\sigma}(t),\lambda^{\sigma}(t),x(T))\Delta
t.
\end{multline*}
\end{remark}

\begin{example}\label{ham:nc}
Consider the problem
\begin{equation}\label{ham:nc:p}
\begin{gathered}
\text{minimize} \quad
\mathcal{L}[x,u]=\int_{0}^{3}(u^{\sigma}(t))^2+t^{2}(x(3)-1)^2
\Delta
     t ,\\
x^{\Delta}(t)=u^{\sigma}(t),
 \end{gathered}
\end{equation}
\begin{equation}
\label{ham:nc:bc}
    x(0)=0,\,  \quad (x(3) \text{ free}).
\end{equation}
To find candidate solutions for the problem, we start by forming the
Hamiltonian function
\begin{equation*}
H(t,x^{\sigma},u^{\sigma},\lambda^{\sigma},x(3))=(u^{\sigma})^2+t^{2}(x(3)-1)^2+\lambda^{\sigma}u^{\sigma}.
\end{equation*}
Candidate solutions $(\tilde{x}(\cdot),\tilde{u}(\cdot))$ are those
satisfying the following conditions:
\begin{equation}\label{ham:nc:h1}
(\lambda^{\sigma}(t))^{\Delta}=0,
\end{equation}
\begin{equation}\label{ham:nc:2}
u^{\sigma}(t)=x^{\Delta}(t), \quad x(0)=0,
\end{equation}
\begin{equation}\label{ham:nc:3}
2u^{\sigma}(t)+\lambda^{\sigma}(t)=0,
\end{equation}
\begin{equation}\label{ham:nc:h4}
\lambda^{\sigma}(\rho(3))=\int_{0}^{3}2t^{2}(x(3)-1)\Delta t.
\end{equation}
From \eqref{ham:nc:h1}--\eqref{ham:nc:3} we conclude that
$\tilde{x}(t)=ct$. In order to determine $c$ we use the
transversality condition \eqref{ham:nc:h4} which we can write as
\begin{equation}\label{ham:nc:h5}
-2c=\int_{0}^{3}2t^{2}(3c-1)\Delta t.
\end{equation}
The value of the delta integral in \eqref{ham:nc:h5} depends on the time scale. Let us compute, for example,
this delta integral on $\mathbb{T}=\mathbb{Z}$
and on $\mathbb{T}= q^{\mathbb{N}_0}$ for $q = 2$.
For $\mathbb{T}=\mathbb{Z}$
\begin{equation}\label{ham:nc:di}
\int_{0}^{3}2t^{2}(3c-1)\Delta
t=\sum_{k=0}^{2}2k^{2}(3c-1)=10(3c-1).
\end{equation}
Equations \eqref{ham:nc:h5} and \eqref{ham:nc:di} yield
$c=\frac{5}{16}$. Therefore, the extremal of the problem
\eqref{ham:nc:p}--\eqref{ham:nc:bc} on $\mathbb{T}=\mathbb{Z}$ is
$\tilde{x}(t)=\frac{5}{16}t$. On the other hand, for
$\mathbb{T}= 2^{\mathbb{N}_0}$ we have
\begin{equation*}
\int_{0}^{3}2t^{2}(3c-1)\Delta
t=2\left(1- 3 c \right)\sum_{t \in \{1, 2\}} t^3=18(1-3c) \, ,
\end{equation*}
and $\tilde{x}(t)=\frac{9}{26}t$.
\end{example}

When $\mathbb{T} = \mathbb{R}$ we immediately obtain
from Theorem~\ref{thm:mr:ham} the following corollary.

\begin{corollary}
\label{Cor:T=R:Hamilt}
If $(\tilde{x}(\cdot),\tilde{u}(\cdot))$ is a solution of the
problem
\begin{equation*}
\begin{gathered}
\text{minimize} \quad \mathcal{L}[x,u]
= \int_a^T f(t,x(t),u(t),x(T)) dt \\
x'(t)=g(t,x(t),u(t),x(T))\\
x(a) = \alpha \quad (x(T) \text{ free}),
\end{gathered}
\end{equation*}
where $\alpha$, $a$, $T \in \mathbb{R}$, $a<T$, then there exists
a function $\tilde{\lambda}(\cdot)$
such that the triple
$(\tilde{x}(\cdot),\tilde{u}(\cdot),\tilde{\lambda}(\cdot))$
satisfies the Hamiltonian system
\begin{equation*}
x'(t) = H_{\lambda} \, , \quad \lambda'(t)=-H_{x} \, ,
\end{equation*}
the stationary condition
\begin{equation*}
H_{u}=0,
\end{equation*}
for all $t \in [a,T]$, and the transversality condition
\begin{equation}
\label{eq:tc}
\lambda(T) =\int_{a}^{T}H_{z} dt,
\end{equation}
where the Hamiltonian $H$ is defined by
$H(t,x,u,\lambda,z)=f(t,x,u,z)+\lambda g(t,x,u,z)$.
\end{corollary}

\begin{remark}
In the classical context $f$ and $g$ do not depend
on $x(T)$. In that case the transversality
condition \eqref{eq:tc} coincides with the standard one ($\lambda(T) = 0$) and Corollary~\ref{Cor:T=R:Hamilt}
coincides with the Hestenes theorem \cite{Hestenes}
(a weak form of the Pontryagin maximum principle \cite{MR0166037}). We were not able to find a single reference to the transversality condition \eqref{eq:tc} in the literature
of optimal control.
\end{remark}

We illustrate the use of Corollary~\ref{Cor:T=R:Hamilt} with an example:

\begin{example}\label{ham:nc:R}
Consider the problem
\begin{equation}\label{ham:nc:rp}
\begin{gathered}
\text{minimize} \quad \mathcal{L}[x,u]=\int_{-1}^{1}(u(t))^2dt ,\\
x'(t)=u(t)+x(1)t,
 \end{gathered}
\end{equation}
\begin{equation}
\label{ham:nc:rbc}
    x(-1)=1  \quad (x(1) \text{ free}).
\end{equation}
We begin by writing the Hamiltonian function
\begin{equation*}
H(t,x,u,\lambda,x(1))=u^2+\lambda(u+x(1)t).
\end{equation*}
Candidate solutions $(\tilde{x}(\cdot),\tilde{u}(\cdot))$ are those
satisfying the following conditions:
\begin{equation}\label{ham:nc:rh1}
\lambda'(t)=0,
\end{equation}
\begin{equation}\label{ham:nc:r2}
x'(t)=u(t)+x(1)t, \quad x(-1)=1,
\end{equation}
\begin{equation}\label{ham:nc:r3}
2u(t)+\lambda(t)=0,
\end{equation}
\begin{equation}\label{ham:nc:rh4}
\lambda(1)=\int_{-1}^{1}\lambda(t)tdt.
\end{equation}
The equation \eqref{ham:nc:rh1} has solution $\tilde{\lambda}(t)=c$,
$-1\leq t \leq 1$, which upon substitution into \eqref{ham:nc:rh4}
yields
\begin{equation*}
c=\int_{-1}^{1}ctdt=0 \, .
\end{equation*}
From the stationary condition \eqref{ham:nc:r3} we get
$\tilde{u}(t)=0$. Finally, substituting the optimal control
candidate back into \eqref{ham:nc:r2} yields
\begin{equation*}
\tilde{x}'(t)=\tilde{x}(1)t.
\end{equation*}
Integrating the latter equation with the initial condition
$\tilde{x}(-1)=1$, we obtain
\begin{equation}\label{ham:nc:r5}
\tilde{x}(t)=\frac{1}{2}\tilde{x}(1)t^{2}+1-\frac{1}{2}\tilde{x}(1).
\end{equation}
Substituting $t=1$ into \eqref{ham:nc:r5} we get $\tilde{x}(1)=1$.
Therefore, the candidate to solution for the problem
\eqref{ham:nc:rp}--\eqref{ham:nc:rbc} is
$\tilde{x}(t)=\frac{1}{2}t^{2}+\frac{1}{2}$.
\end{example}

In certain cases it is easy to show that the extremal candidates obtained from Theorem~\ref{thm:mr}, Theorem~\ref{thm:mr:ham},
or one of the corollaries are indeed minimizers.

\begin{theorem}\label{sc}
Let $(x^{\sigma},u^{\sigma},z)\rightarrow
f(t,x^{\sigma},u^{\sigma},z)$ be jointly convex in
$(x^{\sigma},u^{\sigma},z)$ and
$(x^{\sigma},u^{\sigma},z)\rightarrow g(t,x^{\sigma},u^{\sigma},z)$
be linear in $(x^{\sigma},u^{\sigma},z)$ for each $t$.
If $(\tilde{x}(\cdot),\tilde{u}(\cdot),\tilde{\lambda}(\cdot))$
is a solution of system \eqref{ham2}--\eqref{ham4}, then
$(\tilde{x}(\cdot),\tilde{u}(\cdot))$ is a global minimizer
of \eqref{ocp}--\eqref{ocbc:2}.
\end{theorem}

\begin{proof}
Since $f$ is convex in $(x^{\sigma},u^{\sigma},z)$ for any
admissible pair $(x(\cdot),u(\cdot))$, we have
\begin{multline*}
\mathcal{L}[x,u]-\mathcal{L}[\tilde{x},\tilde{u}]=\int_{a}^{T}\left[f(t,x^{\sigma}(t),u^{\sigma}(t),x(T))
-f(t,\tilde{x}^{\sigma}(t),\tilde{u}^{\sigma}(t),\tilde{x}(T))\right]\Delta
t\\
\geq
\int_{a}^{T}\Bigl[f_{x^{\sigma}}(t,\tilde{x}^{\sigma}(t),
\tilde{u}^{\sigma}(t),\tilde{x}(T))(x^{\sigma}(t)-\tilde{x}^{\sigma}(t))\\
+f_{u^{\sigma}}(t,\tilde{x}^{\sigma}(t),\tilde{u}^{\sigma}(t),\tilde{x}(T))(u^{\sigma}(t)-\tilde{u}^{\sigma}(t))\\
+f_{z}(t,\tilde{x}^{\sigma}(t),\tilde{u}^{\sigma}(t),\tilde{x}(T))(x(T)-\tilde{x}(T))\Bigr]\Delta t.
\end{multline*}
Because the triple
$(\tilde{x}(\cdot),\tilde{u}(\cdot),\tilde{\lambda}(\cdot))$
satisfies equations \eqref{ham1},\eqref{ham3}, and \eqref{ham4}, we
obtain
\begin{multline*}
\mathcal{L}[x,u]-\mathcal{L}[\tilde{x},\tilde{u}] \geq \int_{a}^{T}\Bigl[-\tilde{\lambda}^{\sigma}(t)g_{x^{\sigma}}(\cdots)(x^{\sigma}(t)-\tilde{x}^{\sigma}(t))\\
-(\tilde{\lambda}^{\sigma}(t))^{\Delta}(x^{\sigma}(t)
-\tilde{x}^{\sigma}(t))-\tilde{\lambda}^{\sigma}(t)
g_{u^{\sigma}}(\cdots)(u^{\sigma}(t)-\tilde{u}^{\sigma}(t))\\
-\tilde{\lambda}^{\sigma}(t)g_{z}(\cdots)(x(T)-\tilde{x}(T))\Bigr]\Delta
t\\
+\left[\tilde{\lambda}^{\sigma}(\rho(T))-\int_{\rho(T)}^{T}H_{x^{\sigma}}(\cdots\cdot)\Delta
t\right](x(T)-\tilde{x}(T)),
\end{multline*}
where
$(\cdots)=(t,\tilde{x}^{\sigma}(t),\tilde{u}^{\sigma}(t),\tilde{x}(T))$
and $(\cdots\cdot)=
\left(t,\tilde{x}^{\sigma}(t),\tilde{u}^{\sigma}(t),\tilde{\lambda}^{\sigma}(t),\tilde{x}(T)\right)$.
Integrating by parts the term in
$(\tilde{\lambda}^{\sigma}(t))^{\Delta}$  we get
\begin{multline*}
\mathcal{L}[x,u]-\mathcal{L}[\tilde{x},\tilde{u}]\geq
\int_{a}^{T}\Bigl[-\tilde{\lambda}^{\sigma}(t)g_{x^{\sigma}}(\cdots)(x^{\sigma}(t)-\tilde{x}^{\sigma}(t))\\
+\tilde{\lambda}^{\sigma}(t)(x^{\Delta}(t)-\tilde{x}^{\Delta}(t))-\tilde{\lambda}^{\sigma}(t)
g_{u^{\sigma}}(\cdots)(u^{\sigma}(t)-\tilde{u}^{\sigma}(t))
-\tilde{\lambda}^{\sigma}(t)g_{z}(\cdots)(x(T)-\tilde{x}(T))\Bigr]\Delta t\\
+\left[\tilde{\lambda}^{\sigma}(\rho(T))-
\int_{\rho(T)}^{T}H_{x^{\sigma}}(\cdots\cdot)\Delta
t\right](x(T)-\tilde{x}(T))-\tilde{\lambda}^{\sigma}(T)(x(T)-\tilde{x}(T)).
\end{multline*}
But from \eqref{ham1} and properties of the delta integral,
we have $$\tilde{\lambda}^{\sigma}(T)=\tilde{\lambda}^{\sigma}(\rho(T))-
\int_{\rho(T)}^{T}H_{x^{\sigma}}(\cdots\cdot)\Delta t \, .$$
Rearranging the terms we obtain:
\begin{multline}
\label{ine:bc:sc}
\mathcal{L}[x,u]-\mathcal{L}[\tilde{x},\tilde{u}]\geq
\int_{a}^{T}\Bigl\{-\tilde{\lambda}^{\sigma}(t)\bigl[\tilde{x}^{\Delta}(t)-x^{\Delta}(t)
+g_{x^{\sigma}}(\cdots)(x^{\sigma}(t)-\tilde{x}^{\sigma}(t))\\
+g_{u^{\sigma}}(\cdots)(u^{\sigma}(t)-\tilde{u}^{\sigma}(t))
+g_{z}(\cdots)(x(T)-\tilde{x}(T))\bigr]\Bigr\}\Delta t.
\end{multline}
The intended conclusion follows
from \eqref{ham2} and linearity of $g$ in $(x^{\sigma},u^{\sigma},z)$:
the right hand side of inequality \eqref{ine:bc:sc}
is equal to zero, that is,
\begin{equation*}
\mathcal{L}[x,u]\geq\mathcal{L}[\tilde{x},\tilde{u}]
\end{equation*}
for each admissible pair $(x(\cdot),u(\cdot))$.
\end{proof}

\begin{example}
Consider the problem \eqref{ham:nc:rp}--\eqref{ham:nc:rbc} in Example
\ref{ham:nc:R}. The integrand is independent of $(x,z)$ and convex
in $u$. The right-hand side of the control system is linear in
$(u,z)$ and independent of $x$. Hence, Theorem~\ref{sc}
asserts that the extremal
\begin{eqnarray*}
\tilde{x}(t)=\frac{1}{2}\left(t^{2}+1\right)\\
\tilde{u}(t)=0,\quad \tilde{\lambda}(t)=0,
\end{eqnarray*}
found in Example~\ref{ham:nc:R} gives the
global minimum to the problem.
\end{example}

\begin{example}
Consider again the problem from Example \ref{fromB}. Replacing
$x^{\Delta}$ by $u^{\sigma}$ we can rewrite problem \eqref{AD}--\eqref{eq:bc} in the following form:
\begin{equation*}
 \text{minimize} \quad   \mathcal{L}[x]=\int_{0}^{1}\left(\sqrt{1+(u^{\sigma}(t))^2}+\beta(x(1)-1)^2 \right)\Delta
     t
\end{equation*}
subject to
\begin{equation*}
x^{\Delta}(t)=u^{\sigma}(t),\,  \quad x(0)=0  \quad (x(1)
\text{ free}).
\end{equation*}
The function $f$ is independent of $x$ and convex in $(u,z)$. The
right-hand side of the control system is linear in $u$ and
independent of $(x,z)$. Therefore, by Theorem~\ref{sc}
$\tilde{x}(t) = \alpha(\beta) t$
is the global minimizer for the problem.
\end{example}


\section*{Acknowledgements}

The first author was supported by Bia{\l}ystok Technical University, via a project of the Polish Ministry of Science
and Higher Education ``Wsparcie miedzynarodowej mobilnosci
naukowcow''; the second author by the R\&D unit CEOC, via FCT
and the EC fund FEDER/POCI 2010. ABM is grateful to the good
working conditions at the Department of Mathematics
of the University of Aveiro where this research was carried out.




\begin{thebibliography}{99}

\bibitem{ts:survey} (1908825)
R. Agarwal, M. Bohner, D. O'Regan\ and\ A. Peterson,
\emph{Dynamic equations on time scales: a survey},
J. Comput. Appl. Math. {\bf 141} (2002), no.~1-2, 1--26.

\bibitem{b1} (1786081)
C. D. Ahlbrandt, M. Bohner\ and\ J. Ridenhour,
\emph{Hamiltonian systems on time scales},
J. Math. Anal. Appl. {\bf 250} (2000), no.~2, 561--578.

\bibitem{cv:02} (1908827)
C. D. Ahlbrandt\ and\ C. Morian,
\emph{Partial differential equations on time scales},
J. Comput. Appl. Math. {\bf 141} (2002), no.~1-2, 35--55.

\bibitem{Almeida:T} R. Almeida\ and\ D. F. M. Torres,
Isoperimetric problems on time scales with nabla derivatives,
J. Vib. Control, in press, \arXiv:0811.3650

\bibitem{Atici:Uysal:06} (2218315)
F. M. Atici, D. C. Biles\ and\ A. Lebedinsky,
An application of time scales to economics,
Math. Comput. Modelling {\bf 43} (2006), no.~7-8, 718--726.

\bibitem{Atici:Uysal:08} (2433734)
F. M. Atici\ and\ F. Uysal,
A production-inventory model of HMMS on time scales,
Appl. Math. Lett. {\bf 21} (2008), no.~3, 236--243.

\bibitem{Au:Hilger} (1062633)
B. Aulbach\ and\ S. Hilger,
A unified approach to continuous and discrete dynamics,
in {\it Qualitative theory of differential equations (Szeged, 1988)}, Colloq. Math. Soc. J\'anos Bolyai, 53, North-Holland, Amsterdam, 1990, 37--56.

\bibitem{Bang:q-calc} (2026931)
G. Bangerezako,
\emph{Variational $q$-calculus},
J. Math. Anal. Appl. {\bf 289} (2004), no.~2, 650--665.

\bibitem{ZbigDel} (2445270)
Z. Bartosiewicz\ and\ D. F. M. Torres,
\emph{Noether's theorem on time scales},
J. Math. Anal. Appl. {\bf 342} (2008), no.~2, 1220--1226.

\bibitem{b7} (2106410)
M. Bohner, \emph{Calculus of variations on time scales},
Dynam. Systems Appl. {\bf 13} (2004), no.~3-4, 339--349.

\bibitem{bhn:Gus} (2332777)
M. Bohner\ and\ G. Sh. Guseinov,
\emph{Double integral calculus of variations on time scales}, Comput. Math. Appl. {\bf 54} (2007), no.~1, 45--57.

\bibitem{BP} (1843232)
M. Bohner\ and\ A. Peterson,
``Dynamic equations on time scales",
Birkh\"auser Boston, Boston, MA, 2001.

\bibitem{book:ts1} (1962542)
M. Bohner\ and\ A. Peterson,
``Advances in dynamic equations on time scales",
Birkh\"auser Boston, Boston, MA, 2003.

\bibitem{dt:qc}
T. Ernst,
\emph{The different tongues of $q$-calculus},
Proc. Estonian Acad. Sci. {\bf 57} (2008), no.~2, 81--99.

\bibitem{RuiDel07} (2405376)
R. A. C. Ferreira\ and\ D. F. M. Torres,
\emph{Remarks on the calculus of variations on time scales},
Int. J. Ecol. Econ. Stat. {\bf 9} (2007), no.~F07, 65--73.

\bibitem{Rui:Del:HO}
R. A. C. Ferreira\ and\ D. F. M. Torres,
\emph{Higher-order calculus of variations on time scales},
in ``Mathematical Control Theory and Finance",
Springer, Berlin, 2008, pp.~149--159.

\bibitem{GelfandFomin} (0160139)
I. M. Gelfand\ and\ S. V. Fomin,
``Calculus of variations",
Prentice Hall, Englewood Cliffs, N.J., 1963.

\bibitem{Hestenes} (0203540)
M. R. Hestenes,
``Calculus of variations and optimal control theory'',
Wiley, New York, 1966.

\bibitem{HZ1} (2020533) R. Hilscher\ and\ V. Zeidan,
\emph{Calculus of variations on time scales:
weak local piecewise $C\sp 1\sb {\rm rd}$ solutions
with variable endpoints},
J. Math. Anal. Appl. {\bf 289} (2004), no.~1, 143--166.

\bibitem{HZ2}
R. Hilscher\ and\ V. Zeidan,
\emph{Weak maximum principle and accessory
problem for control problems on time scales},
Nonlinear Analysis: Theory, Methods \& Applications,
in press, DOI: 10.1016/j.na.2008.04.025.

\bibitem{Jackson} (1506108)
F. H. Jackson,
\emph{$q$-Difference Equations},
Amer. J. Math. {\bf 32} (1910), no.~4, 305--314.

\bibitem{QC} (1865777)
V. Kac\ and\ P. Cheung,
``Quantum calculus'', Springer, New York, 2002.

\bibitem{KP} (1142573)
W. G. Kelley\ and\ A. C. Peterson,
``Difference equations",
Academic Press, Boston, MA, 1991.

\bibitem{1st:book:ts} (1419803)
V. Lakshmikantham, S. Sivasundaram\ and\ B. Kaymakcalan,
``Dynamic systems on measure chains'',
Kluwer Acad. Publ., Dordrecht, 1996.

\bibitem{Mal:Tor:09}
A. B. Malinowska\ and\ D. F. M. Torres,
\emph{Necessary and sufficient conditions
for local Pareto optimality on time scales},
J. Math. Sci. (N. Y.), 2009, in press, \arXiv:0801.2123

\bibitem{MR0166037} (0166037) L.~S.~Pontryagin,
V.~G.~Boltyanskii,  R.~V.~Gamkrelidze, E.~F.~Mishchenko,
``The mathematical theory of optimal processes'',
Interscience Publishers John Wiley \& Sons, Inc.\,New York-London, 1962.

\end{thebibliography}
\end{document}